\documentclass[a4paper,11pt]{amsart}
\usepackage[margin=3cm]{geometry}

\usepackage[T1]{fontenc}
\usepackage{lmodern, amsfonts,amsmath,amstext,amsbsy,amssymb,
amsopn,amsthm,upref,eucal,mathptmx,mathtools,url}

\usepackage{mathrsfs}

\RequirePackage{xcolor} 
\definecolor{halfgray}
{gray}{0.55}
\definecolor{webgreen}
{rgb}{0,0.4,0}
\definecolor{webbrown}
{rgb}{.8,0.1,0.1}
\definecolor{red}
{rgb}{1,0,0}
\usepackage{microtype}
\usepackage{tikz}

\newcommand \R {{ \mathbb R}}
\def\C{{\mathbb C}}

\newcommand \Z {{ \mathbb Z}}

\newcommand*{\diff}{\mathop{}\!\mathrm{d}}

\newcommand{\norm}[1]{\left\lVert#1\right\rVert}

\newcommand{\SL}{%
\operatorname{SL}
}

\DeclareMathOperator{\vol}{vol}
\DeclareMathOperator{\un}{u}
\DeclareMathOperator{\Spec}{Spec}
\DeclareMathOperator{\Ad}{Ad}

\newtheorem{theorem}{Theorem}
\newtheorem {lemma}[theorem]{Lemma}
\newtheorem {proposition}[theorem]{Proposition}
\newtheorem{corollary}[theorem]{Corollary}

\date{\today}

\author{Davide Ravotti}

\address{
Monash University, School of Mathematics \\ Clayton Campus, 3800 Victoria, Australia
}

\email{davide.ravotti@gmail.com\\}

 \title[Quantitative equidistribution of horocycle push-forwards of transverse arcs]
 {Quantitative equidistribution of horocycle push-forwards of transverse arcs}

\begin{document}
\maketitle

\begin{abstract}
Let $M = \Gamma \backslash \SL(2,\R)$ be a compact quotient of $\SL(2,\R)$ equipped with the normalized Haar measure $\vol$, and let $\{h_t\}_{t \in \R}$ denote the horocycle flow on $M$.
Given $p \in M$ and $W \in \mathfrak{sl}_2(\R) \setminus \{0\}$ not parallel to the generator of the horocycle flow, let $\gamma_{p}^W$ denote the probability measure uniformly distributed along the arc $s \mapsto p \exp(sW)$ for $0\leq s \leq 1$. 
We establish quantitative estimates for the rate of convergence of $[(h_t)_\ast \gamma_{p}^W](f)$ to $\vol(f)$ for sufficiently smooth functions $f$.
Our result is based on the work of Bufetov and Forni \cite{BuFo}, together with a crucial geometric observation.
As a corollary, we provide an alternative proof of Ratner's theorem on quantitative mixing for the horocycle flow.
\end{abstract}

\section{Introduction}

The horocycle flow $\{h_t\}_{t \in \R}$ on compact quotients of $\SL(2,\R)$ is one of the fundamental examples of parabolic unipotent flows.
Its dynamical and ergodic properties are well-understood: it has zero entropy \cite{Gur}, it is minimal \cite{Hed}, uniquely ergodic \cite{Fur}, mixing and mixing of all orders \cite{Ma2}, and has countable Lebesgue spectrum \cite{Par}. 

Quantitative versions of these results have also been investigated. 
Ratner \cite{Rat1} established optimal polynomial mixing rates for H{\" o}lder observables. Moreover, it follows from a general result by Bj{\" o}rklund, Einsiedler, and Gorodnik \cite{BEG} that, for all $k \geq 2$, mixing of order $k$ is also polynomial.
Regarding quantitative equidistribution, Flaminio and Forni \cite{FF} proved precise results on the asymptotics of ergodic averages of smooth functions. The results in \cite{FF} have been refined by Bufetov and Forni in \cite{BuFo}, where the authors construct a family of finitely-additive H{\" o}lder measures and associated H{\" o}lder functionals which control the asymptotics of ergodic integrals. The main result of this paper, Theorem \ref{thm:main} below, is based on the properties of these \emph{Bufetov-Forni functionals}, which we recall in~\S\ref{sec:BF}.

\subsection{Shearing properties of the horocycle flow}

One of the key geometric properties of the horocycle flow is a form of \emph{shearing} of transverse arcs.
For example, let $\gamma(s) = p \exp(sX)$, $0 \leq s \leq 1$, be the geodesic segment starting at $p \in M$ of length 1. From the usual commutation relation between the geodesic and horocycle flow, it is easy to see that the curve $h_t \circ \gamma$ is \emph{sheared} along the direction of the horocycle flow and, for $t$ large enough, it approximates a long segment of an orbit of the horocycle flow. 
In particular, it becomes \emph{equidistributed}; namely, given any continuous function $f$, the integral of $f \circ h_t$ along $\gamma(s)$ converges to the space average of $f$, when $t \to \infty$.  
This mechanism has been exploited in the proofs of several results; for example, by Marcus in \cite{Ma1} to prove mixing for horocycle flows in variable negative curvature, and by Forni and Ulcigrai in \cite{FU} to establish Lebesgue spectrum of smooth time-changes of the standard horocycle flow.

Moreover, it follows from the work of Bufetov and Forni \cite{BuFo} that such sheared geodesic arcs equidistribute at the same rate as horocycle orbits. 
It is also plain that the same phenomenon happens if one replace the initial geodesic arc with any homogeneous segment which lies in a single $\{X,U\}$-leaf; that is, a single weak-stable leaf for the geodesic flow. 

What happens if the arc $\gamma(s) = p \exp(sW)$ is not contained in a single $\{X,U\}$-leaf (i.e., if the generator $W \in \mathfrak{sl}_2(\R)$ has a non-zero component in the direction of the unstable horocycle flow) is less clear, since the curve  $h_t \circ \gamma$ \emph{does not} approximate a single orbit of the horocycle flow.
By approximating it with \emph{several} orbit segments, it is possible to show that also in this case the curve $h_t \circ \gamma$ equidistributes. However, the quantitative estimates one can prove with this approach are far from optimal. 
In this paper, following a different strategy which heavily relies on the properties of the Bufetov-Forni functionals (in particular, on the \emph{unstable horocycle invariance}, see Lemma \ref{key}), we prove sharper estimates for the equidistribution of arbitrary sheared arcs, see Theorem \ref{thm:main}. As a corollary, we provide an alternative proof of Ratner's quantitative mixing estimates, see Corollary~\ref{thm:main2}.

\subsection{Definitions and notations}

Before stating our main result, we recall some definitions and we fix the notation.

Let $\Gamma < \SL(2,\R)$ be a co-compact lattice, let $M = \Gamma \backslash \SL(2,\R)$, and let us denote by $\vol$ the smooth probability measure on $M$ given locally by the Haar measure.
The manifold $M$ can be identified with the unit tangent bundle $T^1S$ of the compact hyperbolic surface $S = \Gamma \backslash \mathbb{H}$. 
The spectrum of the \emph{Laplace-Beltrami operator} $\Delta_S$ on $S$ is pure point and discrete. In particular, if we denote by $(\mu_n)_{n \geq 0}$ the positive eigenvalues, there is a \emph{spectral gap}, since the bottom $\mu_0$ of the non-zero spectrum is strictly positive.
Let us further define
$$
\nu_0 :=
\begin{cases}
\sqrt{1-4\mu_0} & \text{ if } \mu_0 < 1/4,\\
0 & \text{ if } \mu_0 \geq 1/4,
\end{cases}
\text{\ \ \ \ \ \ }
\varepsilon_0 :=
\begin{cases}
0 & \text{ if } \mu_0 \neq 1/4,\\
1 & \text{ if } \mu_0 = 1/4.
\end{cases}
\text{\ \ \ and\ \ \ }
\delta_0 :=
\begin{cases}
0 & \text{ if } \mu_0 < 1/4,\\
1 & \text{ if } \mu_0 \geq 1/4.
\end{cases}
$$

Let us denote by $L^2(M, \vol)$ the space of square-integrable functions and by $L^2_0(M, \vol)$ the subspace of functions with zero integral.
For any $r>0$, let $W^r(M)$ be the Sobolev space of functions $f \in L^2(M, \vol)$ such that $\Delta^{r/2}f \in L^2(M,\vol)$, where $\Delta$ is a Laplacian on $M$ (see \S\ref{sec:spectral} for definitions).
Let us remark that, for $r > 3/2$, by Sobolev Embedding Theorem, we have that $W^r(M) \subset \mathscr{C}^{\alpha}(M)$, for $\alpha < r-3/2$.
 
We denote the Lie algebra of $\SL(2,\R)$ by $\mathfrak{sl}_2(\R)$; it consists of $2 \times 2$ real matrices with zero trace.
Each element $W \in \mathfrak{sl}_2(\R) \setminus \{0\}$ generates the homogeneous flow $\{\varphi_t^W \}_{t \in \R}$ on $M$ given by 
$$
\varphi_t^W(\Gamma g) = \Gamma g \exp(tW).
$$ 
We fix the basis $\mathscr{B} = \{V, X, U\}$ of $\mathfrak{sl}_2(\R)$, where 
$$
V =
\begin{pmatrix}
0 & 0 \\
1 & 0
\end{pmatrix},
\ \ \ 
X =
\begin{pmatrix}
1/2 & 0 \\
0 & -1/2
\end{pmatrix},
\ \ \ 
U =
\begin{pmatrix}
0 & 1 \\
0 & 0
\end{pmatrix}.
$$
Then, $\mathscr{B}$ is a frame of the tangent bundle of $M$.
The homogeneous flows generated by $V,X,U$ are the unstable horocycle flow  $ \{h^{\un}_t \}_{t \in \R}$, the geodesic flow $ \{g_t \}_{t \in \R}$, and the stable horocycle flow $\{h_t \}_{t \in \R}$ respectively.
Finally, let us denote by $\widehat{\mathscr{B}} = \{\widehat{V}, \widehat{X}, \widehat{U} \}$ the frame of the cotangent bundle of $M$ dual to $\mathscr{B}$.

\subsection{Statement of the main result}

Let $f \in L^2_0(M,\vol)$ be a sufficiently smooth function with zero average. 
We are interested in the asymptotics of the following integrals
$$
\left\lvert \int_0^\sigma f \circ h_t \circ \varphi^W_s(p) \diff s \right\rvert,
$$
for fixed $\sigma >0$ and $W \in \mathfrak{sl}_2(\R) \setminus \{0\}$.

The case $W=X$ follows easily from the work of Bufetov and Forni; a more general statement for smooth time-changes of the horocycle flow can be found in \cite[Lemma 17]{FU}.

\begin{theorem}
Let $r >11/2$. For any $\sigma >0$, there exists a constant $C=C(r,\sigma)>0$ such that for any $f \in W^r(M) \cap L^2_0(M,\vol)$, for all $0 \leq S \leq \sigma$, $t\geq 1$ and $p \in M$ we have 
$$
\left\lvert \int_0^S f \circ h_t \circ g_s(p) \diff s \right\rvert \leq C \norm{f}_r t^{-\frac{1-\nu_0}{2}} (\log t)^{\varepsilon_0}.
$$
\end{theorem}

It is easy to see that the same estimates (up to constant) hold if the geodesic segment $\{g_s(p) : s \in [0,S]\}$ is replaced with any segment of the form $\{\varphi_s^W(p): s \in [0,S]\}$, where $W = xX + uU$, with $x \neq 0$.

In this paper, we generalize the previous result to \emph{arbitrary} $W \in \mathfrak{sl}_2(\R) \setminus \{0\}$, namely to arbitrary arcs $\{\varphi_s^W(p) : s \in [0,S] \}$ not tangent to the integrable distribution $\{X,U\}$. 
Our main result is the following.

\begin{theorem}\label{thm:main}
Let $r > 11/2$ and let
$$
W = vV+xX+uU \in \mathfrak{sl}_2(\R), \text{\ \ \ with\ \ \ }v \neq 0.
$$
For any $\sigma >0$, there exists a constant $C=C(r,\sigma,W)$ such that for any $f \in W^r(M)\cap L^2_0(M,\vol)$, for all $0 \leq S \leq \sigma$, $t \geq 1$ and $p \in M$, we have 
$$
\left\lvert \int_0^S f \circ h_t \circ \varphi^W_s(p) \diff s \right\rvert \leq C \norm{f}_r t^{-(1-\nu_0)}(\log t)^{\delta_0}.
$$
\end{theorem}

In particular, the result above shows that arcs with a non-zero component in the direction $V$ equidistribute \emph{faster} than geodesic arcs.

The proof of Theorem \ref{thm:main} follows two different strategies for functions supported on the principal and complementary series, and for functions supported on the discrete series, see Propositions \ref{prop1} and \ref{prop2}.
In the former case, treated in Proposition \ref{prop1}, we exploit the properties of the H{\" o}lder functionals constructed by Bufetov and Forni, together with a key geometrical observation, see Lemma \ref{key}. 
In the case of negative Casimir parameters, however, the Bufetov-Forni functionals are not defined. For functions supported on the discrete series, we then proceed by a standard approximation argument, analogous to the case of the push-forward of geodesic arcs.
For this reason, the bounds achieved in this case, contained in Proposition \ref{prop2}, are not optimal. 
However, the resulting estimates for generic functions obtained in Theorem \ref{thm:main} by combining Propositions \ref{prop1} and \ref{prop2} differ from the optimal ones only by a logarithmic factor, and only when $\mu_0 >1/4$.

By a standard \emph{mixing via shearing} argument, from equidistribution estimates for the push-forward of transverse arcs, it is possible to deduce quantitative mixing estimates. Optimal mixing rates for the geodesic and horocycle flows were obtained by Ratner in \cite{Rat1}; here, we provide an alternative geometric proof of her result\footnote{As explained before, our estimates differ from Ratner's by a factor $\log t$ when $\mu_0 > 1/4$.} for smooth observables.

\begin{corollary}[{Ratner's quantitative mixing \cite{Rat1}}]\label{thm:main2}
Let $r > 11/2$. There exists a constant $C>0$ such that for all $f \in W^r(M)$, for all $g \in L^{2}(M,\vol)$ such that $Vg \in L^{2}(M,\vol)$, and for all $t \geq 1$ we have 
$$
\left\lvert \int_M (f \circ h_t) \, g \diff \vol - \int_M f \diff \vol \int_M g \diff \vol \right\rvert \leq C \norm{f}_r (\norm{g}_2 + \norm{Vg}_{2}) t^{-(1-\nu_0)}(\log t)^{\delta_0}.
$$
\end{corollary}

\section{Preliminaries}

\subsection{Shearing of transverse arcs}

Let us recall that, if we denote by $L_g \colon \SL(2,\R) \to \SL(2,\R)$ the left-multiplication by $g$ and by $(L_g)_\ast$ its push-forward, any element $W \in \mathfrak{sl}_2(\R)$ induces the vector field, which we still denote by $W$, defined by $W|_g=(L_g)_\ast (W)$.

We will need the following basic facts about the tangent vectors of curves in $M$. 

\begin{lemma}\label{derivative}
Let $W = vV+xX+uU \in \mathfrak{sl}_2(\R) \setminus \{0\}$ and let
$$
\exp(sW) = 
\begin{pmatrix}
a(s) & b(s) \\
c(s) & d(s)
\end{pmatrix}.
$$
Then
$$
\begin{pmatrix}
a'(s) & b'(s) \\
c'(s) & d'(s)
\end{pmatrix}
=
\begin{pmatrix}
a(s) & b(s) \\
c(s) & d(s)
\end{pmatrix}
\begin{pmatrix}
x/2 & u \\
v & -x/2
\end{pmatrix}.
$$
\end{lemma}
\begin{proof}
The claim follows by the chain rule: let $g = \exp(sW)$; then we have
\begin{equation*}
\begin{split}
\begin{pmatrix}
a'(s) & b'(s) \\
c'(s) & d'(s)
\end{pmatrix}
&= \frac{\diff}{\diff s} \exp(sW) = (L_g)_\ast \left( \frac{\diff}{\diff s} \bigg\rvert_{s=0} \exp(sW) \right) = (L_g)_\ast(W) \\
&= \begin{pmatrix}
a(s) & b(s) \\
c(s) & d(s)
\end{pmatrix}
\begin{pmatrix}
x/2 & u \\
v & -x/2
\end{pmatrix}
.
\end{split}
\end{equation*}
\end{proof}

For any $g \in \SL(2,\R)$, the Adjoint $\Ad_g \colon \mathfrak{sl}_2(\R) \to \mathfrak{sl}_2(\R)$ is the linear map defined by $\Ad_g(W) = g^{-1}Wg$. The Adjoint describes the action of a homogeneous flow on tangent vectors, namely we have the following result.

\begin{lemma}\label{adjoint}
Let $W,Y \in \mathfrak{sl}_2(\R) \setminus \{0\}$. For all $t,s \in \R$ and for all $p \in M$, we have
$$
\frac{\diff}{\diff s} \varphi_t^Y \circ \varphi_s^W(p) =  \Ad_{\exp(tY)}(W)|_{\varphi_t^Y \circ \varphi_s^W(p)}
$$
\end{lemma}
\begin{proof}
We have 
\begin{equation*}
\begin{split}
\frac{\diff}{\diff s} \varphi_t^Y \circ \varphi_s^W(p) &= \frac{\diff}{\diff s}  \left(p \exp(sW) \exp(tY) \right) \\
&= \left( L_{p \exp(sW) \exp(tY)} \right)_{\ast} \left( \frac{\diff}{\diff s} \bigg\rvert_{s=0} \exp(-tY)\exp(sW)\exp(tY) \right) \\
&=  \left( L_{\varphi_t^Y \circ \varphi_s^W(p)} \right)_{\ast} \Ad_{\exp(tY)}(W).
\end{split}
\end{equation*}
\end{proof}

\subsection{Spectral theory of the horocycle flow}\label{sec:spectral}

We now recall some results about the spectral theory of the horocycle flow.
Let 
$$
\Theta = 
\begin{pmatrix}
0 & 1/2 \\
-1/2 & 0
\end{pmatrix}
$$
be a generator of the maximal compact subgroup $K = \text{SO}(2) \subset \SL(2,\R)$. 
We define the \emph{Casimir operator} $\square$ by
$$
\square = -X^2 -(V + \Theta)^2 +\Theta^2,
$$
which is a generator of the centre of the universal enveloping algebra of $\mathfrak{sl}_2(\R)$. 
By the classical theory of unitary representations of $\SL(2,\R)$, we have the following orthogonal decomposition into irreducible components (listed with multiplicity)
\begin{equation}\label{eq:suml2}
L^2(M) = \bigoplus_{\mu \in \Spec(\square)} H_\mu,
\end{equation}
and on each $H_\mu$, the Casimir operator $\square$ acts as the constant $\mu \in \R_{>0} \cup \{ -n^2+n : n \in \Z_{\geq 0}\}$. The component $H_0$ corresponds to the trivial representation and appears with multiplicity one. The positive Casimir parameters $\mu >0$ coincide with the positive eigenvalues of the Laplace-Beltrami operator $\Delta_S$ on the hyperbolic surface $S = \Gamma \backslash \mathbb{H}$.

The irreducible components $H_\mu$ are divided into three series: the \emph{principal series} for Casimir parameters $\mu \geq 1/4$,  the \emph{complementary series} for Casimir parameters $0 <\mu < 1/4$, and  the \emph{discrete series} for negative Casimir parameters $\mu \leq 0$.

Let $\Delta $ be the \emph{Laplacian} defined by $\Delta = -(X^2 + U^2/2 + V^2/2)$. 
For any Hilbert space $H$ on which $\SL(2,\R)$ acts unitarily, we define the \emph{Sobolev space $W^r(H)$ of order} $r>0$ to be the maximal domain of the operator $(1+\Delta)^{r/2}$, equipped with the inner product 
$$
\langle f,g \rangle_r = \langle (1+\Delta)^r f,g \rangle_{H}.
$$
We denote by $\norm{\cdot}_r$ the norm in $W^r(H)$ defined by the inner product above. 
When $H = L^2(M,\vol)$, we will simply write $W^r(M)$ instead of $W^r(L^2(M,\vol))$.
The space $W^r(M)$ coincides with the completion of $\mathscr{C}^{\infty}(M)$ of infinitely differentiable functions on $M$ with respect to the norm $\norm{\cdot}_r$.
The direct sum \eqref{eq:suml2} induce a corresponding decomposition of the Sobolev spaces for any fixed $r>0$; namely we have 
\begin{equation}\label{eq:sumw}
W^r(M) = \bigoplus_{\mu \in \Spec(\square)} W^r(H_\mu).
\end{equation}

\subsection{Bufetov-Forni functionals}\label{sec:BF}

The proof of our main result is a consequence of the properties of the Bufetov-Forni functionals introduced in \cite{BuFo}. We briefly recall the results we will use, and we refer the reader to \cite{BuFo} and to the work of Flaminio and Forni \cite{FF} for more details.

For every positive Casimir parameter $\mu>0$, let $\nu = \sqrt{1-4\mu} \in \C$. We remark that $\nu$ is purely imaginary if $\mu$ belongs to the principal series, and $0<\nu<1$ if $\mu$ belongs to the complementary series.

\begin{theorem}[{\cite[Theorem 1.1]{BuFo}}]\label{thm:BF}
For any positive Casimir parameter $\mu >0$, there exist two independent normalized finitely additive measures $\widehat{\beta}^{\pm}_{\mu}$ such that for any rectifiable arc $\gamma \subset M$ the following properties hold:
\begin{enumerate}
\item for any decomposition $\gamma = \gamma_1 + \gamma_2$ into subarcs, 
$$
\widehat{\beta}^{\pm}_{\mu}(\gamma) = \widehat{\beta}^{\pm}_{\mu}(\gamma_1) + \widehat{\beta}^{\pm}_{\mu}(\gamma_2);
$$
\item for all $t \in \R$ and $\mu\neq 1/4$,
$$
\widehat{\beta}^{\pm}_{\mu}(g_{-t}\gamma) = \exp\left( \frac{1\mp \nu}{2} t\right) \widehat{\beta}^{\pm}_{\mu}(\gamma),
$$
while, for $\mu = 1/4$,
$$
\begin{pmatrix}
\widehat{\beta}^{+}_{1/4}(g_{-t}\gamma) \\
\widehat{\beta}^{-}_{1/4}(g_{-t}\gamma)
\end{pmatrix}
= \exp \left( \frac{t}{2}\right) 
\begin{pmatrix}
1 & -\frac{t}{2} \\
0 & 1
\end{pmatrix}
\begin{pmatrix}
\widehat{\beta}^{+}_{1/4}(\gamma) \\
\widehat{\beta}^{-}_{1/4}(\gamma)
\end{pmatrix}
$$
\item for all $t \in \R$, 
$$
\widehat{\beta}^{\pm}_{\mu}(h^{\un}_{t}\gamma) = \widehat{\beta}^{\pm}_{\mu}(\gamma),
$$
\item there exists a constant $C >0$ such that for all $\mu \neq 1/4$,
$$
\left\lvert \widehat{\beta}^{\pm}_{\mu}(\gamma) \right\rvert \leq C \left( 1+ \int_\gamma | \widehat{X} | + \int_\gamma | \widehat{U}| \int_\gamma |\widehat{V}| \right) \left( \int_\gamma |\widehat{U}| \right)^{\frac{1\mp \Re \nu}{2}}
$$
and, for $\mu = 1/4$,
\begin{equation*}
\begin{split}
\left\lvert \widehat{\beta}^{+}_{1/4}(\gamma) \right\rvert &\leq C \left( 1+ \int_\gamma | \widehat{X} | + \int_\gamma | \widehat{U}| \int_\gamma |\widehat{V}| \right) \left( \int_\gamma |\widehat{U}| \right)^{\frac{1}{2}} \left( 1 + \log\int_\gamma |\widehat{U}| \right)\\
\left\lvert \widehat{\beta}^{-}_{1/4}(\gamma) \right\rvert &\leq  C \left( 1+ \int_\gamma | \widehat{X} | + \int_\gamma | \widehat{U}| \int_\gamma |\widehat{V}| \right) \left( \int_\gamma |\widehat{U}| \right)^{\frac{1}{2}}
\end{split}
\end{equation*}
\end{enumerate}
\end{theorem}

For any $\mu>0$, let us denote by $D^+_\mu, D^-_\mu$ the two normalized invariant distributions introduced by Flaminio and Forni in \cite{FF}.
Let $r > 11/2$; for any function $f \in W^r(M)$ supported on irreducible components of the principal and complementary series and for any rectifiable arc $\gamma \subset M$, let us define 
$$
 \widehat{\beta}_{f}(\gamma)= \sum_{\mu \in \Spec (\square) \cap \R_{>0}} D^+_\mu(f)  \widehat{\beta}^{+}_{\mu} (\gamma) + D^-_\mu(f)  \widehat{\beta}^{-}_{\mu}(\gamma).
$$
\begin{theorem}[{\cite[Theorem 1.3]{BuFo}}]\label{thm:BF2}
For any $r > 11/2$ there exists a constant $C_r >0$ such that for every rectifiable arc $\gamma \subset M$ and for all $f \in W^r(M)$ supported on irreducible components of the principal and complementary series we have
$$
\left\lvert \int_\gamma f \widehat{U} -  \widehat{\beta}_{f}(\gamma) \right\rvert \leq C_r\norm{f}_r \left( 1+ \int_\gamma |\widehat{X}| + \int_\gamma |\widehat{V}| \right).
$$
\end{theorem}

\section{Proof of Theorem \ref{thm:main}}

Let $W = vV+xX+uU \in \mathfrak{sl}_2(\R)$ with $v \neq 0$ be fixed. 
Let $\sigma >0$ and $0 \leq S \leq \sigma$, and define the curve
$$
\gamma_{p,t}^W(s) = h_t \circ \varphi_s^{W}(p), \text{\ \ \ for } s \in [0, S].
$$
Up to replacing $\sigma$ and $S$ by $c\sigma$ and $cS$, for some $c>0$, we can assume that $\max\{|v|,|x|,|u|\} \leq 1$.

Let $r > 11/2$ and let $f \in W^r(M) \cap L^2_0(M,\vol)$. Using the orthogonal decomposition \eqref{eq:sumw}, we can write 
$$
f = \sum_{\mu \in \Spec (\square)} f_{\mu},
$$
where $f_\mu \in W^r(H_\mu)$, and, for $\mu=0$, the component $f_0=0$.
Let us denote by
$$
f_{d} = \sum_{\mu \in \Spec (\square) \cap \R_{<0}} f_{\mu}, \text{\ \ \ and\ \ \ } f_d^{\perp} = \sum_{\mu \in \Spec (\square) \cap \R_{>0}} f_{\mu},
$$
so that 
$$
\int_0^S f \circ \gamma_{p,t}^W(s) \diff s = \int_0^S f_d \circ \gamma_{p,t}^W(s) \diff s + \int_0^S f_d^\perp \circ \gamma_{p,t}^W(s) \diff s. 
$$
We will estimate the two integrals separately. The proof of Theorem \ref{thm:main} follows from Propositions \ref{prop1} and \ref{prop2} below applied to $f_d^\perp$ and $f_d$ respectively.

\subsection{Positive Casimir parameters}

%
This section is devoted to the proof of the follwing result.
\begin{proposition}\label{prop1}
Let $r > 11/2$ and let $f \in W^r(M) \cap L^2_0(M,\vol)$ be supported on irreducible components of the principal and complementary series. Then, there exists a constant $C=C(r,\sigma,W)>0$ such that for all $0 \leq S \leq \sigma$, $t \geq 2$ and $p \in M$, we have
$$
\left\lvert \int_0^S f \circ h_t \circ \varphi_s^W(p) \diff s \right\rvert \leq C \norm{f}_r t^{-(1-\nu_0)}(\log t)^{\varepsilon_0}.
$$ 
\end{proposition}

By Lemma \ref{adjoint}, we can compute the tangent vector of $\gamma_{p,t}^W(s)$ by
\begin{equation}\label{eq:tangent}
\frac{\diff}{\diff s} \gamma_{p,t}^W(s) = \Ad_{\exp(tU)}(W) |_{\gamma_{p,t}^W(s)} = \big(vV + (x-2tv)X + (u+xt-vt^2)U\big)|_{\gamma_{p,t}^W(s)}
\end{equation}
so that
$$
\int_0^S f \circ \gamma_{p,t}^W(s) \diff s =  (u+xt-vt^2)^{-1} \int_{\gamma_{p,t}^W} f \widehat{U}.
$$
For all $t \geq 2$ we have $|u+xt-vt^2| \geq |v|t^2(1-t^{-1}-t^{-2}) \geq |v|t^2/4$. 
Therefore, by Theorem \ref{thm:BF2} and formula \eqref{eq:tangent}, 
\begin{equation}\label{eq:tool1}
\begin{split}
\left\lvert \int_0^S f \circ \gamma_{p,t}^W(s) \diff s \right\rvert & \leq \frac{4}{|v|t^2}\left\lvert \int_{\gamma_{p,t}^W} f \widehat{U} \right\rvert \leq \frac{4}{|v|t^2} \left( \widehat{\beta}_{f}(\gamma_{p,t}^W) + C_r \norm{f}_r S (1+|v|+|x|t+2|v|t)\right) \\
& \leq C_{r,W, \sigma} \left( \frac{1}{t^2} \widehat{\beta}_{f}(\gamma_{p,t}^W) + \frac{\norm{f}_r}{t} \right),
\end{split}
\end{equation}
where we can take $C_{r,W,\sigma}= \max \{4|v|^{-1}, 20 C_r \sigma |v|^{-1} \}$.

Recall that we have
$$
\widehat{\beta}_{f}(\gamma_{p,t}^W) =  \sum_{\mu \in \Spec (\square) \cap \R_{>0}} D_{\mu}^+(f)\widehat{\beta}^+_{\mu}(\gamma_{p,t}^W) + D_{\mu}^-(f)\widehat{\beta}^-_{\mu}(\gamma_{p,t}^W).
$$
From properties (2) and (3) in Theorem \ref{thm:BF}, for any $\mu \neq 1/4$, we deduce that for all $T \in \R$,
$$
\widehat{\beta}^{\pm}_{\mu}(\gamma_{p,t}^W) = t^{1\mp \nu}\widehat{\beta}^{\pm}_{\mu}( g_{2\log t} \circ\gamma_{p,t}^W) = t^{1\mp \nu}\widehat{\beta}^{\pm}_{\mu}( h^{\un}_{T} \circ g_{2\log t} \circ\gamma_{p,t}^W), 
$$
while, for $\mu = 1/4$, we have
\begin{equation*}
\begin{split}
\widehat{\beta}^{+}_{\mu}(\gamma_{p,t}^W) &=t \widehat{\beta}^+_{\mu}( g_{2\log t} \circ\gamma_{p,t}^W) + t \log t\widehat{\beta}^-_{\mu}( g_{2\log t} \circ\gamma_{p,t}^W) \\
&=t \widehat{\beta}^+_{\mu}(  h^{\un}_{T} \circ g_{2\log t} \circ\gamma_{p,t}^W) + t \log t\widehat{\beta}^-_{\mu}(  h^{\un}_{T} \circ g_{2\log t} \circ\gamma_{p,t}^W), 
\end{split}
\end{equation*}
and
$$
\widehat{\beta}^{-}_{\mu}(\gamma_{p,t}^W) =t \widehat{\beta}^-_{\mu}( g_{2\log t} \circ\gamma_{p,t}^W) =t \widehat{\beta}^-_{\mu}(  h^{\un}_{T} \circ g_{2\log t} \circ\gamma_{p,t}^W).
$$

The following is our key geometrical observation: choosing $T=-t$ in the formulas above, the components of the tangent vectors of the curve $h^{\un}_{-t} \circ g_{2\log t} \circ \gamma_{p,t}^W$ with respect to the frame $\mathscr{B}$ are \emph{uniformly bounded} in $t$.
\begin{lemma}\label{key}
We have
$$
\frac{\diff}{\diff s} (h^{\un}_{-t} \circ g_{2\log t} \circ \gamma_{p,t}^W) (s) = -uV - \left( x+\frac{2u}{t} \right) X + \left( \frac{u}{t^2} + \frac{x}{t} - v\right) U,
$$
in particular for all $t \geq 1$ we have
$$
\int_{h^{\un}_{-t} \circ g_{2\log t} \circ \gamma_{p,t}^W} |\widehat{Y}| \leq 3S \leq 3\sigma, \text{\ \ \ for any\ }\widehat{Y} \in \widehat{\mathscr{B}} = \{\widehat{V},\widehat{X},\widehat{U}\}.
$$
\end{lemma}
\begin{proof}
By Lemma \ref{adjoint}, from an easy computation it follows
\begin{equation*}
\begin{split}
&\frac{\diff}{\diff s} (h^{\un}_{-t} \circ g_{2\log t} \circ \gamma_{p,t}^W) (s) = \Ad_{\exp(-tV)} \circ \Ad_{\exp(2\log t X)} \circ \Ad_{\exp(tU)} (W)\bigg\rvert_{h^{\un}_{-t} \circ g_{2\log t} \circ \gamma_{p,t}^W(s)} \\
& \qquad = \Ad_{\exp(-tV)} \circ \Ad_{\exp(2\log t X)} (vV + (x-2vt)X + (-vt^2+xt+u)U)\bigg\rvert_{h^{\un}_{-t} \circ g_{2\log t} \circ \gamma_{p,t}^W(s)}\\
& \qquad = \Ad_{\exp(-tV)} (vt^2V + (x-2vt)X + (-v+xt^{-1}+ut^{-2})U)\bigg\rvert_{h^{\un}_{-t} \circ g_{2\log t} \circ \gamma_{p,t}^W(s)}\\
& \qquad = (-uV - (x+2ut^{-1})X + (-v+xt^{-1}+ut^{-2})U)\bigg\rvert_{h^{\un}_{-t} \circ g_{2\log t} \circ \gamma_{p,t}^W(s)}.
\end{split}
\end{equation*}
In particular, for all $t\geq 1$, we have $\max \{|u|, |x|+ 2|u|t^{-1}, |v| + |x|t^{-1} + |u|t^{2} \} \leq 3$, which concludes the proof.
\end{proof}
By property (4) in Theorem \ref{thm:BF} and by Lemma \ref{key}, there exists a constant $C_{\sigma} >0$ such that for all $\mu \neq 1/4$ we have 
$$
|\widehat{\beta}^{\pm}_{\mu}(\gamma_{p,t}^W)| \leq C_{\sigma} t^{1\mp \Re\nu}.
$$
and, for $\mu=1/4$,
$$
|\widehat{\beta}^{\pm}_{1/4}(\gamma_{p,t}^W)| \leq C_{\sigma} t(1+ \log t).
$$
Up to enlarging $C_\sigma$, for all $t \geq 1$ we then deduce
$$
\left\lvert \widehat{\beta}_{f}(\gamma_{p,t}^W) \right\rvert \leq  C_{\sigma} \norm{f}_r t^{1+ \nu_0}(\log t)^{\varepsilon_0},
$$
which, together with \eqref{eq:tool1}, concludes the proof.

\subsection{The discrete series}

In this section we prove the follwing result.
\begin{proposition}\label{prop2}
Let $r > 11/2$ and let $f \in W^r(M) \cap L^2_0(M,\vol)$ be supported on irreducible components of the discrete series. Then, there exists a constant $C=C(r,\sigma,W)>0$ such that for all $0 \leq S \leq \sigma$, $t \geq 2$ and $p \in M$, we have
$$
\left\lvert \int_0^S f \circ h_t \circ \varphi_s^W(p) \diff s \right\rvert \leq C \norm{f}_r t^{-1}(\log t).
$$ 
\end{proposition}

Using the same notation as in Lemma \ref{derivative}, let
$$
\exp(sW) = 
\begin{pmatrix}
a(s) & b(s) \\
c(s) & d(s)
\end{pmatrix}.
$$
Since the entries above are smooth functions, we can define 
$$
\ell = \ell(\sigma) = 4 \max\{ |a'(s)|, |b'(s)|, |c'(s)|, |d'(s)| : s \in [0,\sigma] \}.
$$

For negative Casimir parameters, the Bufetov-Forni functionals are not defined. 
Hence, we proceed in the following way: we partition the curve $\gamma_{p,t}^W$ into $O(t)$ curves which are at distance $O(t^{-1})$ from a leaf tangent to the integrable distribution $\{X,U\}$, and we approximate them by their projection onto the $\{X,U\}$-leaf. 
The estimate for each projection can be then deduced from \cite{BuFo}, see e.g. \cite[Lemma 17]{FU}.

We decompose the integral as follows:
\begin{equation}\label{eq:tool2}
\begin{split}
\left\lvert \int_0^S f \circ h_t \circ \varphi_s^W(s) \diff s \right\rvert &= \left\lvert \int_0^S f \circ \gamma_{p,t}^W(s) \diff s \right\rvert \leq \frac{\norm{f}_{\infty}}{\ell t} + \sum_{k=0}^{\lfloor S \ell t \rfloor -1}  \left\lvert \int_{\frac{k}{\ell t}}^{\frac{k+1}{\ell t}} f \circ \gamma_{p,t}^W(s) \diff s \right\rvert \\
&= \frac{\norm{f}_{\infty}}{\ell t} + \sum_{k=0}^{\lfloor S \ell t \rfloor -1}  \left\lvert \int_{0}^{\frac{1}{\ell t}} f \circ \gamma_{p_k, t}^W(s) \diff s \right\rvert,
\end{split}
\end{equation}
where 
$$
p_k = p\, \exp\left( \frac{k}{\ell t} W\right).
$$

For any $s \in [0, (\ell t)^{-1}]$, we have 
$$
|d(s) + c(s)t| \geq |d(s)|-|c(s)|t \geq 1- \frac{\ell}{4}s - \frac{\ell}{4}st \geq 1 - \frac{1}{4t} - \frac{1}{4} \geq \frac{1}{2}.
$$
Therefore, the function
$$
J_0(s) = -\frac{c(s)}{d(s) + c(s)t}
$$
is well-defined  for $s \in [0, (\ell t)^{-1}]$.

For any $0 \leq k \leq \lfloor S \ell t \rfloor -1$, let us define the curve
$$
\overline{\gamma_k} (s) = h^-_{J_0(s)} \circ \gamma_{p_k,t}^W(s), \text{\ \ \ for } s \in  [0, (\ell t)^{-1}].
$$
Explicitly, we have
\begin{equation*}
\begin{split}
\overline{\gamma_k} (s)  &= p_k \, \exp(sW)
\begin{pmatrix}
1 & t \\
0 & 1
\end{pmatrix}
\begin{pmatrix}
1 & 0 \\
J_0(s) & 1
\end{pmatrix}
=
p_k 
\begin{pmatrix}
\frac{1}{d(s)+c(s)t} & b(s)+a(s)t \\
0 & d(s)+c(s)t
\end{pmatrix}
\\
&= 
p_k
\begin{pmatrix}
1 & \frac{b(s)+a(s)t}{d(s)+c(s)t} \\
0 & 1
\end{pmatrix}
\begin{pmatrix}
\frac{1}{d(s)+c(s)t} & 0 \\
0 & d(s)+c(s)t
\end{pmatrix}
= g_{-2\log (d(s)+c(s)t)} \circ h_{\frac{b(s)+a(s)t}{d(s)+c(s)t}}(p_k)
\end{split}
\end{equation*}
Notice that the curve $\overline{\gamma_k}$ is contained in the $\{X,U\}$-leaf passing through $p_k$.
We also remark that there exists a constant $C>0$ such that for any $s \in [0, (\ell t)^{-1}]$,
$$
\text{dist}(\overline{\gamma_k} (s), \gamma_{p_k,t}^W(s)) \leq C \left\lvert J_0(s)\right\rvert \leq 2C|c(s)| \leq 2C\ell s \leq \frac{2C}{t},
$$
so that 
\begin{equation}\label{eq:tool3}
\begin{split}
\sum_{k=0}^{\lfloor S \ell t \rfloor -1}  \left\lvert \int_{0}^{\frac{1}{\ell t}} f \circ \gamma_{p_k,t}^W(s) \diff s \right\rvert & \leq \sum_{k=0}^{\lfloor S \ell t \rfloor -1}  \left\lvert \int_{0}^{\frac{1}{\ell t}} f \circ \overline{\gamma_k} (s) \diff s \right\rvert + \sum_{k=0}^{\lfloor S \ell t \rfloor -1}  \left\lvert \frac{2C\norm{f}_r}{\ell t^2} \right\rvert \\
& \leq \sum_{k=0}^{\lfloor S \ell t \rfloor -1}  \left\lvert \int_{0}^{\frac{1}{\ell t}} f \circ \overline{\gamma_k} (s) \diff s \right\rvert + \frac{2C\norm{f}_{r} \sigma}{t}.
\end{split}
\end{equation}

\begin{lemma}
We have
$$
\widehat{U} \left( \frac{\diff}{\diff s} \overline{\gamma_k} (s) \right) = -vt^2 + \frac{x}{2}t +u,
$$
and moreover
$$
\left\lvert \widehat{X} \left( \frac{\diff}{\diff s} \overline{\gamma_k} (s) \right) \right\rvert \leq 20t.
$$
\end{lemma}
\begin{proof}
Let us denote
$$
J_1(s) = -2\log (d(s)+c(s)t), \text{\ \ \ and\ \ \ }J_2(s) = \frac{b(s)+a(s)t}{d(s)+c(s)t}.
$$
We compute
\begin{equation*}
\begin{split}
\frac{\diff}{\diff s} \overline{\gamma_k} (s) =& \frac{\diff}{\diff s} \left( g_{J_1(s)} \circ h_{J_2(s)}(p_k) \right) = D g_{J_1(s)} \left( \frac{\diff}{\diff s} \circ h_{J_2(s)}(p_k)\right) + \left( \frac{\diff}{\diff s} g_{J_1(s)}\right) \circ h_{J_2(s)}(p_k) \\
=& D g_{J_1(s)} \left( \Big(\frac{\diff}{\diff s} J_2(s) \Big) U \circ h_{J_2(s)}(p_k) \right) + \left( \Big( \frac{\diff}{\diff s} J_1(s) \Big) X \circ g_{J_1(s)}\right) \circ h_{J_2(s)}(p_k) \\
=& \Big(\frac{\diff}{\diff s} J_2(s) \Big) e^{-J_1(s)} U \circ  \overline{\gamma_k} (s)  + \Big( \frac{\diff}{\diff s} J_1(s) \Big) X \circ  \overline{\gamma_k} (s) \\
=& \Big( (b'(s) + a'(s)t)(d(s)+c(s)t) - (b(s)+a(s)t)(d'(s)+c'(s)t)\Big) U \circ  \overline{\gamma_k} (s) \\
&-2\frac{c'(s)t+d'(s)}{c(s)t+d(s)} X \circ  \overline{\gamma_k} (s).
\end{split}
\end{equation*}
By Lemma \ref{derivative}, for any $s \in [0, (\ell t)^{-1}]$, we have
$$
\max\{ c'(s), d'(s) \} \leq 2 \max\{ c(s), d(s) \} \leq 2\left(1+ \frac{\ell}{4} s \right) \leq 3,
$$
so that 
$$
\left\lvert 2 \frac{c'(s)t+d'(s)}{c(s)t+d(s)} \right\rvert \leq 4 (3t+3) \leq 20t.
$$
Using the formulas in Lemma \ref{derivative} and the fact that $a(s)d(s)-b(s)c(s)=1$, we obtain that
$$
(b'(s) + a'(s)t)(d(s)+c(s)t) - (b(s)+a(s)t)(d'(s)+c'(s)t) = -vt^2 + \frac{x}{2}t +u,
$$
which concludes the proof.
\end{proof}

With our assumptions, for $t \geq 2$,
$$
|-vt^2 + \frac{x}{2}t +u| \geq |v|t^2(1-t^{-1}-t^{-2}) \geq \frac{|v|t^2}{4}.
$$
Thus, it follows that (see, e.g., the proof of \cite[Lemma 17]{FU})
\begin{equation}\label{eq:tool4}
\begin{split}
\left\lvert \int_{0}^{\frac{1}{\ell t}} f \circ \overline{\gamma_k} (s) \diff s \right\rvert &=  \left\lvert \left( -vt^2 + \frac{x}{2}t +u\right)^{-1} \int_{\overline{\gamma_k}} f \widehat{U} \right\rvert \leq \frac{4}{|v|t^2} \left\lvert C_r \norm{f}_r \left(1+\int_{\overline{\gamma_k}} |\widehat{X}| \right) \log\left( 1 + \int_{\overline{\gamma_k}} |\widehat{U}| \right) \right\rvert\\
& \leq \frac{16}{|v|t^2} C_r \norm{f}_r \left( 1+ \frac{20}{\ell} \right) (\log t ) \leq C_{r,W,\sigma} \norm{f}_r \frac{\log t}{t^2}.
\end{split}
\end{equation}

From \eqref{eq:tool2}, \eqref{eq:tool3} and \eqref{eq:tool4}, we conclude
$$
\left\lvert \int_0^S f \circ h_t \circ \varphi_s^W(p) \diff s \right\rvert \leq  C_{r,W,\sigma} \norm{f}_r \frac{\log t}{t}.
$$

\section{Proof of Corollary \ref{thm:main2}}

The proof is a straightforward consequence of a standard \lq\lq mixing via shearing\rq\rq\ argument, see, e.g., \cite{FU}. 
Without loss of generality, let us assume that $f \in W^r(M) \cap L^2_0(M, \vol)$ and let $g \in L^2(M, \vol)$ such that $Vg \in L^2(M, \vol)$.
By invariance of the Haar measure and integrating by parts, for all $\sigma >0$ and $t \geq 1$ we have
\begin{equation*}
\begin{split}
\langle f \circ h_t, g \rangle =& \frac{1}{\sigma} \int_0^{\sigma} \langle f \circ h_t \circ h_s^{\un}, g \circ h_s^{\un} \rangle \diff s = \langle \frac{1}{\sigma} \int_0^{\sigma} f \circ h_t\circ h_s^{\un} \diff s, g \rangle \\
& - \frac{1}{\sigma} \int_0^{\sigma} \langle \int_0^S f \circ h_t \circ h_s^{\un} \diff s, Vg \circ h_S^{\un} \rangle,
\end{split}
\end{equation*}
so that 
\begin{equation*}
\begin{split}
|\langle f \circ h_t, g \rangle | & \leq \frac{\norm{g}_{2}}{\sigma} \norm{\int_0^{\sigma} f \circ h_t \circ h_s^{\un} \diff s}_{2} + \frac{\norm{Vg}_{2}}{\sigma} \sup_{S\in[0,{\sigma}]} \norm{ \int_0^S f \circ h_t \circ h_s^{\un} \diff s}_{2} \\
& \leq \frac{\norm{g}_{2} +\norm{Vg}_{2}}{\sigma} \sup_{S\in[0,\sigma]} \sup_{p \in M} \left\lvert \int_0^S f \circ h_t \circ h_s^{\un}(p) \diff s\right\rvert.
\end{split}
\end{equation*}
The result follows from the estimates in Theorem \ref{thm:main}.

\subsection*{Acknowledgements}
I would like to thank Giovanni Forni and Corinna Ulcigrai for several useful discussions.
The research leading to these results has received funding from the European Research Council under the European Union
Seventh Framework Programme (FP/2007-2013) / ERC Grant Agreement n.~335989.


\begin{thebibliography}{20}

\bibitem{BEG} M.~Bj{\" o}rklund, M.~Einsiedler, A.~Gorodnik, {\it Quantitative multiple mixing}, to appear in JEMS, arXiv:1701.00945.

\bibitem{BuFo} A.~Bufetov, G.~Forni, {\it Limit theorem for horocycle flows}, Ann. Sci. {\' E}c. Norm. Sup{\' e}r. \textbf{47}(5):851--903, 2014. 

\bibitem{FF} L.~Flaminio, G.~Forni, {\it Invariant distributions and time averages for horocycle flows},  Duke Math. J. \textbf{119}(3):465--526, 2003.

\bibitem{FU} 
G.~Forni \& C.~Ulcigrai, 
Time-changes of horocycle flows, 
{\it Journal of Modern Dynamics}~\textbf{6} (2) (2012), 251--273.   

\bibitem{Fur} H.~Furstenberg, {\it The unique ergodicity of the horocycle flow}, in Recent Advances in Topological Dynamics (New Haven, Conn., 1972), Lecture Notes in Math. 318, Springer, Berlin, 1973, 95-115.

\bibitem{Ma1} 
B.~Marcus, 
Ergodic properties of horocycle flows for surfaces of negative curvature, 
{\it Ann. of Math.} (2) \textbf{105} (1977), 81--105.

\bibitem{Ma2} 
B.~Marcus, 
The horocycle flow is mixing of all degrees, 
{\it Invent. Math.} (3) \textbf{46} (1978), 201--209.

\bibitem{Gur} B.~M.~Gurevic, {\it The entropy of horocycle flows}, Dokl. Akad. Nauk SSSR \textbf{136}:768--770, 1961. 

\bibitem{Hed} G.~A.~Hedlund, {\it Fuchsian groups and transitive horocycles}, Duke Math. J. \textbf{2}:530--542, 1936. 

\bibitem{Par} O.~S.~Parasyuk, {\it Flows of horocycles on surfaces of constant negative curvature} (in Russian), Uspekhi Mat. Nauk \textbf{8}(3):125--126, 1953.

\bibitem{Rat1} M.~Ratner, {\it The rate of mixing for geodesic and horocycle flows}, Ergodic Theory Dynam. Systems
\textbf{7}:267--288, 1987.
\end{thebibliography}
 \end{document}